\documentclass[letterpaper, 10 pt, conference]{ieeeconf}  

\IEEEoverridecommandlockouts                              
\usepackage{graphics} 
\usepackage{epsfig} 
\usepackage{amsmath} 
\usepackage{amssymb}  
\usepackage{bbm}
\usepackage{eufrak}

\newtheorem{lemma}{Lemma}
\newtheorem{theorem}{Theorem}

\newtheorem{definition}{Definition}

\newtheorem{remark}{Remark}

\usepackage{mathtools}
\usepackage{soul}
\usepackage{tikz}
\usepackage{color}
\usepackage{tcolorbox}
\usepackage{float}
\usepackage{subfig}
\usepackage{algorithm}
\usepackage[noend]{algpseudocode}
\usepackage{tikzscale}
\mathtoolsset{showonlyrefs}
\usepackage{textcomp}

\makeatletter
\def\BState{\State\hskip-\ALG@thistlm}
\makeatother

\usepackage{pgfplots} 
\pgfplotsset{compat=newest} 
\pgfplotsset{plot coordinates/math parser=false} 
\newlength\figureheight 
\newlength\figurewidth 

\title{\LARGE \bf
Controllability and Data-Driven Identification of \\Bipartite Consensus on Nonlinear Signed Networks
}

\author{Mathias Hudoba de Badyn, Siavash Alemzadeh, and Mehran Mesbahi
\thanks{This research was supported by the U.S. Army Research Laboratory
and the U.S. Army Research Office under contract number
W911NF-13-1-0340, NSF grant SES-1541025 and AFOSR grant FA9550-16-1-0022.}
\thanks{The first and last authors are with the William E. Boeing Department of Aeronautics and Astronautics, and the second author is with the Department of Mechanical Engineering, all at the
        University of Washington, Seattle WA, USA
        {\tt\small \{hudomath,alems,mesbahi\}@uw.edu.}~\textcopyright~2017 IEEE}%
}

\begin{document}

\maketitle
\thispagestyle{empty}
\pagestyle{empty}

\begin{abstract}
Nonlinear networked systems are of interest in several areas of research, such as multi-agent systems and social networks.
In this paper, we examine the controllability of several classes of nonlinear networked dynamics on which the underlying graph admits negative weights.
Such signed networks exhibit bipartite clustering when the underlying graph is structurally balanced.
We show that structural balance is the key ingredient inducing uncontrollability when combined with a leader-node symmetry and a certain type of dynamical symmetry.
We then examine the problem of extracting the bipartite structure of such graphs from data using Extended Dynamic Mode Decomposition to approximate the corresponding Koopman operator.


\textit{Index Terms}$-$Nonlinear networks, signed networks, controllability, structural balance, graph symmetry, Koopman operator, extended dynamic mode decomposition

\end{abstract}

\section{INTRODUCTION}


Networked dynamical systems are the cornerstone of many modern technologies, as well as the focus of scientific research in many disciplines.
Some pertinent examples of networked systems are opinion dynamics~\cite{Hegselmann2002}, gene networks~\cite{Yeung2002}, flocking dynamics~\cite{vicsek1995} and autonomous coordinated flight~\cite{Model2000}.
One of the most well-studied networked dynamical systems is the consensus protocol, which is used in many applications including Kalman filtering~\cite{Olfati-Saber2005, HudobadebadynFillt2017}, multi-agent systems~\cite{Chen2013a,Saber2003,hudoba2017control} and robotics~\cite{Joordens2009}.
More recently, the controls community has begun examining consensus on networks admitting both cooperative and antagonistic interactions~\cite{Altafini2013,Altafini2012,Srivastava2011,Alemzadeh2017}.

Previous research in this direction has examined the controllability of consensus networks in both the linear case and non-linear generalizations of consensus~\cite{Mesbahi2010,Aguilar2014,Roy2016}.
The work by Rahmani~\emph{et al.} has shown that symmetries of networks characterized by graph automorphisms about leaders in the network cause uncontrollability~\cite{Rahmani2009a}.
This symmetry technique was generalized by Chapman and Mesbahi, who showed that signed fractional automorphisms provide the necessary and sufficient conditions for characterizing controllability of consensus networks~\cite{Chapman2015a}.
Further work examined methods of generating network topologies, for either performance improvements, as in \cite{Siami2016} and \cite{Zelazo2013}, or those that are controllable for consensus, such as in~\cite{Chapman2014a} and~\cite{Hudobadebadyn2016}.

Consensus algorithms on networks admitting antagonistic interactions were considered by~Altafini~\cite{Altafini2013} and Pan \emph{et al.} \cite{Pan2016a}.
The graph-theoretic property of \emph{structural balance}, used in the study of social networks, was identified as the property inducing bipartite consensus \cite{cartwright1956structural,Akiyama1981,Kaszkurewicz2012}.
Moreover, Clark \emph{et al.} worked on the leader selection problem on signed networks~\cite{Clark2017}.
Roy and Abaid discussed using antagonistic interactions to improve existing consensus-like algorithms~\cite{Roy2016}.

The generalization to nonlinear consensus algorithms has been studied in numerous settings.
Behaviour of nonlinear consensus protocols was considered by Srivastava \emph{et al.}~\cite{Srivastava2011}.
The extension of these consensus protocols to signed networks was studied by Altafini~\cite{Altafini2013}.
Moreover, the generalization of symmetry arguments for controllability was examined by by Aguilar and Gharesifard \cite{Aguilar2014}.

There has been a recent interest in applying data-driven methods to control of networks; one such approach has utilized the Koopman operator.
The Koopman operator provides a dynamical framework in which one considers the propagation of observables of the state, rather than the state itself.
The Koopman operator is linear, even for a non-linear dynamical system, but the trade-off is that the vector space of observables is generally infinite-dimensional~\cite{Budisic2012}.
This formalism lends itself well to a data-driven approach, allowing one to approximate the Koopman operator by collecting data~\cite{Williams2015}. 
Research by Pan \emph{et al.} has looked at identifying the bipartite structure of signed networks using data-driven methods~\cite{Pan2016}, furthering work done by Facchetti \emph{et al.}~\cite{Facchetti2011}, and Harary and Kabell~\cite{Harary1980}.

The contributions of this paper are twofold.
First, we show that the property of structural balance, when combined with symmetries in the underlying graph, as well as certain symmetries of the nonlinear dynamics, causes uncontrollability in the context of the accessibility problem.
In particular, we consider the same network flows studied in~\cite{Srivastava2011,Aguilar2014,Altafini2013}; however we extend the controllability analysis to signed dynamics.
Secondly, we extend the bipartite identification problem considered by Pan \emph{et al.} in \cite{Pan2016} to the case of signed nonlinear consensus networks.
In particular, we use a Koopman operator-theoretic approach alongside Extended Dynamic Mode Decomposition (EDMD) to extract a `Koopman mode' whose sign structure reveals the bipartite structure.

The paper is organized as follows.
In \S\ref{sec:math-prel}, we provide the relevant background on graph theory, nonlinear control and Koopman operator theory required for the discussion in this paper.
Then, the problems considered are outlined in \S\ref{sec:problem-statement}.
In \S\ref{sec:nonl-contr-sign}, we examine the controllability problem and in \S\ref{sec:bipart-ident-with} we consider the bipartite identification.
Relevant examples are shown in \S\ref{sec:examples}, and the paper is concluded in \S\ref{sec:concl-future-works}.



\section{MATHEMATICAL PRELIMINARIES}

\label{sec:math-prel}

A column vector with $n$ elements is referred to as $v\in\mathbb{R}^n$ where $v_i$ or $[v]_i$ both represent the $i$th element in $v$.
The square matrix $N\in\mathbb{R}^{n\times n}$ is \emph{symmetric} if $N^T=N$.
The identity matrix is denoted $I_n$.
For $w\in\mathbb{R}^n$ the matrix $\mathrm{diag}(w)$ is an $n\times n$ matrix with $w$ on its diagonal and zero elsewhere.
We say that $A$ is \emph{similar} to $B$ if there is an invertible matrix $R$ such that $R^{-1}AR=B$.
The unit vector $e_i$ is the column vector with all zero entries except $[e_i]_i=1$.
The column vector of all ones is denoted as $\textbf{1}$.
The \emph{cardinality} of a set $S$ is denoted as $|S|$.
The column space of a matrix $M$ is denoted by $\mathcal{R}(M)$.
We denote the Moore-Penrose pseudoinverse of a matrix $A$ as $A^\dag$.
The function $h$ is \emph{even} if $h(-x)=h(x)$ and is \emph{odd} if $h(-x)=-h(x)$.
The function $f$ is of class $C^r$ if the derivatives $f,f',\dots,f^{(r)}$ exist and are continuous.
The function $g\in C^\infty$, otherwise called \emph{smooth}, has derivatives of all orders. 
Let $F:\mathbb{R}^n\rightarrow\mathbb{R}^n$ be a vector field and let $\varphi:\mathbb{R}^n\rightarrow\mathbb{R}^n$ be a smooth mapping.
Then, $F$ is $\varphi$-\emph{invariant} if $(D\varphi(x))F(x)=F(\varphi(x))$ for all $x\in\mathbb{R}^n$ with $D\varphi(x)$ the Jacobian matrix of $\varphi$ at $x$.
Given a mapping $\gamma:\mathcal{M}\rightarrow\mathcal{M}$, the fixed point set of $\gamma$ is denoted by $\text{Fix}(\gamma)=\{x\in\mathcal{M}|\gamma(x)=x \}$.

\subsection{Consensus Dynamics on Signed Networks}\label{sec:graphss-graphsl}

We follow the standard notation and conventions for graph theory applied to multi-agent systems and consensus, as in \cite{Mesbahi2010}.
Below, we introduce some mathematics relating the ideas of symmetry and signed graphs that we will utilize in this paper.

An \emph{automorphism} of the graph $\mathcal{G}$ is a permutation $\phi:\mathcal{V}\to\mathcal{V}$ of its nodes such that $\phi(i)\phi(j)\in\mathcal{E}$ if and only if $ij\in \mathcal{E}$.
The permutation $\phi$ \emph{induces} a mapping $\varphi:\mathbb{R}^n\to\mathbb{R}^n$ such that $[\varphi(x)]_i = x_{\phi(i)}$.
Let the permutation matrix $ J $ be such that $[ J ]_{ij}=1$ if $\phi(i)=j$ and zero otherwise.
Therefore, the permutation matrix $ J $ is simply the Jacobian matrix of $\varphi$, in that $ J  = D \varphi$.
Thus, $\phi$ \emph{represents the automorphism} of $\mathcal{G}$ if and only if $ J  A(\mathcal{G})=A(\mathcal{G})  J $ (see \cite{Mesbahi2010}).

A \emph{signed graph} $\mathcal{G}_s$ is a graph with both positive and negative weights.
We define the \emph{signed graph Laplacian} as  ${L}_s={D}_s-{A}_s$, where $[{A}_s]_{ij}=\pm{A}_{ij}$, with the degree matrix given by 
${D}_{ii}=\sum_{j\in\mathcal{N}(i)} \left|{A}_{ij}\right|=\sum_{j\in\mathcal{N}(i)}\left|{W}_{ij}\right|$.
Following the generalization of consensus to signed graphs in~\cite{Altafini2012}, the signed consensus dynamics are $\dot x = -{L}_sx$, i.e.,
$\dot x_i = -\sum_{j\in\mathcal{N}(i)} \left|{W}_{ij}\right|\left(x_i - \mathrm{sgn}({W}_{ij})x_j\right)$, 
where $\mathrm{sgn}$ indicates the sign function.
A \emph{gauge transformation} is a change of orthant order via a matrix $G_t \in \left\{ \mathrm{diag}(\sigma):~\sigma=[\sigma_1,\dots,\sigma_n]~,~\sigma_i = \pm 1\right\}$.
From the definition $G_t=G_t^T=G_t^{-1}$.
Let the \emph{gauge-transformed Laplacian} be given by ${L}_{G_t}= G_t{L}_sG_t={D} - G_t{A}_sG_t$ where
\begin{align}
\left({L}_{G_t}\right)_{ij} &= 
\begin{cases}
\sum_{k\in \mathcal{N}(i)} |{A}_{ik}|& j=i \\
-\sigma_i\sigma_jA_{ij} & j \neq i.
\end{cases}
\end{align}
In a similar way that we defined the functions $\phi$ and $\varphi$ for the graph automorphism, consider the function $\mathfrak{g}:\mathcal{V}\to\mathcal{V}$ encoding the action of the gauge on the nodes.
This induces a function $g:\mathbb{R}^n\to\mathbb{R}^n$ such that $[g(x)]_i = \sigma_i x_i$.
The gauge transformation is then the Jacobian of this function, in that $G_t = D g$.
\subsection{Nonlinear Dynamical Systems} 

We follow the same conventions as in \cite{Aguilar2014}.
Consider the controlled dynamical system
\begin{align}
	\label{eq:1}
	\dot{x}=F(x,u)
\end{align}
where $F:\mathbb{R}^n\times\mathbb{R}^m\rightarrow\mathbb{R}^n$ is a smooth mapping.
The \emph{accessible set} $\mathcal{A}(x_0,T)$ of the system \eqref{eq:1} from $x_0$ at time $T$ is the set of all end-points $\theta(T)$ where $\theta:[0,T]\rightarrow\mathbb{R}^n$ is a trajectory of \eqref{eq:1}.
The accessible set of \eqref{eq:1} from $x_0$ up to $T$ is defined as $\mathcal{A}(x_0,\leq T)\coloneqq\cup_{0\leq\tau\leq T}\mathcal{A}(x_0,\tau)$.
Then, the nonlinear dynamical system \eqref{eq:1} is said to be \emph{accessible} from the initial point $x_0$ if for every $T>0$ the set $\mathcal{A}(x_0,\leq T)$ contains a non-empty interior.


\subsection{Koopman Operator}

Consider the dynamical system in \eqref{eq:1} without control.
The \emph{Koopman operator} $\mathcal{K}$ acts on functions of the state space (called \emph{observables}) $\psi$ by the action $\mathcal{K}\psi=\psi \circ F$. 
The function $\varphi(x)$ is an \emph{eigenfunction} of $\mathcal{K}$ corresponding to the \emph{eigenvalue} $\mu$ if $\mathcal{K}\varphi(x)=e^{\mu t}\varphi(x)$.
For an observable function $g$ in the span of Koopman eigenfunctions, one can write $g(x)=\sum_{k=1}^{\infty} v_k\varphi_k(x)$, where the $v_k$'s are the \emph{Koopman modes}. 
For the case of full-state observable $g(x)=x$ the states can be reconstructed as, $x(t)=g(x(t))=\sum_{k=1}^{\infty} e^{\mu_k t}\varphi_k(x_0) v_k$, where we refer to $\{\mu_k,\varphi_k, v_k\}$ as the \emph{Koopman triple}.
A numerical approximation of Koopman operator can be obtained by EDMD~\cite{Williams2015}.


\section{PROBLEM STATEMENT}
\label{sec:problem-statement}

In this paper, we tackle two problems regarding nonlinear signed consensus networks.
In the first section, we extend the result in \cite{Aguilar2014} to the signed networks.
We examine the necessary conditions of uncontrollability in nonlinear signed network systems due to input and dynamics symmetry. 
In particular, we will show how the additional topological property of \emph{structural balance} in signed networks plays a key role in driving uncontrollability.

The following lemma from~\cite{Altafini2013} elucidates some properties of structural balance.


\begin{lemma}
	\label{lem:1}
	(See \cite{Altafini2013})
	The following statements are equivalent:
	\begin{enumerate}
		\item The signed graph $\mathcal{G}$ is structurally balanced;
		\item There exists a gauge transformation $G_t$ such that $G_tA_sG_t$ has only positive entries;
		\item For all cycles in $\mathcal{G}$, the product of the edge weights on the cycle are positive;
		\item The signed Laplacian $L_s$ has a zero eigenvalue;
		\item There exists a bipartition of $\mathcal{V}$ such that the edge weights on the edges within the same set are positive, and the edges connecting the two sets are negative.
	\end{enumerate}
\end{lemma}


In the second section of this paper, we show that a particular \emph{Koopman mode} from the EDMD approximation of the Koopman operator contains the sign structure corresponding to the bipartition in Lemma~\ref{lem:1}(5).
This extends the work by Pan \emph{et al.} who considered the equivalent problemn for linear signed consensus~\cite{Pan2016}.



\section{Nonlinear Controllability of Signed Networks}
\label{sec:nonl-contr-sign}

In this section, we extend previous work \cite{Aguilar2014} to analyze the controllability of nonlinear consensus protocols to the case where these protocols run on a signed network.
We consider three types of nonlinear consensus protocols, following the nomenclature in \cite{Altafini2013,Srivastava2011,Aguilar2014}:
\newcommand{\sgn}{{\mathrm{sgn}}}
\begin{itemize}
	\item \textbf{Absolute Nonlinear Flow}
	\begin{align}
	\dot x_i =- \sum_{j\in N_i} \left[f(x_i) - \sgn(a_{ij}) f(x_j)\right]
	\label{eq:5}
	\end{align}
	\item \textbf{Relative Nonlinear Flow}
	\begin{align}
	\dot x_i =- \sum_{j\in N_i} f\left(x_i - \sgn(a_{ij}) x_j\right)
	\label{eq:6}
	\end{align}
	\item \textbf{Disagreement Nonlinear Flow}
	\begin{align}
	\dot x_i =-  f\left(\sum_{j\in N_i} x_i - \sgn(a_{ij}) x_j\right)
	\label{eq:7}
	\end{align}
\end{itemize}

To make the paper self-contained, we provide two main theorems from \cite{Aguilar2014} which we use later to demonstrate uncontrollability.

%


\begin{theorem}
	\label{thm:3}
	Let $\mathcal{G}=(\mathcal{V},\mathcal{E})$ and $F:\mathbb{R}^n\rightarrow\mathbb{R}^n$ be a flow on $\mathcal{G}$. Assume $\varphi$ is a non-identity symmetry on $F$.
        Then, for any leader $l$, the leader-follower network flow on $\mathcal{G}$ induced by $l$ is not accessible from the origin in $\mathbb{R}^{n-1}$.
\end{theorem}


\begin{theorem}
	\label{thm:1}
	Let $\mathcal{G}=(\mathcal{V},\mathcal{E})$ and let $F:\mathbb{R}^n\rightarrow\mathbb{R}^n$ be the dynamics in any of \eqref{eq:5}-\eqref{eq:7}.
        Also, assume $\varphi$ be an automorphism of $\mathcal{G}$.
        Then $F$ is $\varphi$-invariant.
\end{theorem}



The behavior of the dynamics \eqref{eq:5}-\eqref{eq:7} clearly depends on the choice of $f:\mathbb{R}\to\mathbb{R}$.
In \cite{Altafini2013}, several classes of functions were considered.
First, the class of \emph{translated positive, infinite sector nonlinearities} $\mathcal{S}$ is defined as
\begin{align}
\mathcal{S} := \Big\{ f:&[f(x)-f(x^*)](x - x^*) >0~\text{for}~x\neq x^*\Big\},\label{eq:3}
\end{align}
where $\int_{x^*}^x f(t)dt \to \infty\text{ as }|x|\to\infty$ and $f(0)=0$;
see \cite{Kaszkurewicz2012} for properties of this class of functions.
A subset $\mathcal{S}_0\subset \mathcal {S}$ of these functions that will be used later is the \emph{untranslated positive, infinite sector nonlinearities} given by setting $x^*=0$ in the definition of $\mathcal{S}$.
The reason these classes of functions are interesting is that when combined with the dynamics introduced in \eqref{eq:5}-\eqref{eq:7}, \emph{clustering} occurs in a structurally balanced graph.
This is summarized in the following theorem.


\begin{theorem}
\label{thr:2}
	(\emph{Theorems 3 \& 4 in \cite{Altafini2013}})
	{Consider a graph $\mathcal{G}$.
	Assume either the dynamics \eqref{eq:5} with $f\in\mathcal{S}$ or the dynamics \eqref{eq:6} with $f\in\mathcal{S}_0$ running on $\mathcal{G}$.
        Then,
	$\lim_{t\to\infty}x(t) = \frac{1}{n} \left(\mathbf{1}^T G_tx(0)\right) G_t \mathbf{1}$
	if and only if $\mathcal{G}$ is structurally balanced (with gauge transformation $G_t$).}	
\end{theorem}


According to this theorem, for certain classes of functions, the dynamics will converge to two different clusters.
These clusters are exactly those corresponding to the bipartite consensus condition in Lemma~\ref{lem:1}(5).

In the following subsections, we elaborate on the controllability of the dynamics \eqref{eq:5}-\eqref{eq:7} and show that a notion of symmetry about the input node, as well as structural balance, lead to uncontrollability.
From \cite{Altafini2013} we know if the underlying signed graph $\mathcal{G}$ is structurally balanced, then there exists a gauge transformation $G_t$ that acts as a similarity transformation on the adjacency matrix of $\mathcal{G}$ in that $G_t A_s(\mathcal{G})G_t = A$ where $A$ is the unsigned adjacency matrix.
We will show that $G_t$ defines a useful coordinate transformation that allows an immediate application of the uncontrollability test in \cite{Aguilar2014}.

First, we discuss the notion of symmetry in signed graphs.
\begin{definition}
	\label{def:1}
	Let $\varphi$ be a non-identity automorphism on graph $\mathcal{G}$.
	Suppose that this graph is structurally balanced, with gauge transformation $G_t$ induced by the function $g:\mathbb{R}^n\to\mathbb{R}^n$ defined by $[g(x)]_i = \sigma_i x_i$.
	Then, we define \emph{signed automorphism} operator as $\varphi'=g\circ\varphi\circ g$. 
	Moreover, assume that $J$ is the matrix representation of the permutation operator $\varphi$, in that $J=D\varphi$.
	Then, the analogous matrix $J'=G_tJG_t$ is the matrix representation of the signed permutation operator $\varphi'$, in that $J'=D(g\circ \varphi \circ g)$.
\end{definition}
By Definition \ref{def:1},  if $\varphi(x_i)=x_r$, then $\varphi'(x_i)=\sigma_i\sigma_r x_r$. 
For example, consider the graph in Figure~\ref{fig:1}.
\begin{figure}[H]
	\centering
	\includegraphics[scale=0.65]{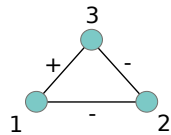}
	\caption{Example of a signed automorphism}
	\label{fig:1}
\end{figure}
The corresponding gauge transformation and automorphisms are defined as
\begin{gather*}
G_t=\begin{bmatrix}
1&0&0\\
0&-1&0\\
0&0&1
\end{bmatrix}
\ , \ 
J=\begin{bmatrix}
0&1&0\\
1&0&0\\
0&0&1
\end{bmatrix}\\
J'=G_tJG_t=\begin{bmatrix}
0&-1&0\\
-1&0&0\\
0&0&1
\end{bmatrix}.
\end{gather*}
This implies that $\varphi'(x_1,x_2,x_3)=(-x_2,-x_1,x_3)$.
Recall that we use $\phi(i)$ as the action of the automorphism on the index of a node rather than the more obscure notation $\varphi(v_i)$.
For example, in Figure \ref{fig:1} we have $\phi(1)=2,~\phi(2)=1$.
We can now proceed to the main results of this paper, which are established case-by-case for each of the dynamics in Equations~\eqref{eq:5}-\eqref{eq:7}.

\subsubsection{Absolute Nonlinear Flow}


In the following theorem, we will show that for absolute nonlinear flow with odd functions $f$, structural balance directly generalizes the uncontrollability conditions in \cite{Aguilar2014}.
For the case of even functions, we need to impose additional topological structure on the edge weights of the underlying graph.


\begin{theorem}
	\label{thm:4}
	Consider a structurally balanced graph $\mathcal{G}$ with gauge transformation $G_t$ and \emph{absolute nonlinear flow} dynamics~\eqref{eq:5}.
	Further suppose $\mathcal{G}$ has a non-trivial signed automorphism $\varphi'$.
	Let $f:\mathbb{R}^n\to \mathbb{R}$ be a smooth odd function (for example, odd $f\in\mathcal{S}_0$ with $f$ smooth).
	Then, for any vertex $j\in\mathrm{Fix}(\varphi')$ chosen as the leader, the leader-follower network is not accessible from the origin in $\mathbb{R}^{n-1}$.
	Moreover, the same results holds for smooth even functions $f$ if $\varphi'$ preserves edge signs, in that $\mathrm{sgn}(a_{ij})=\mathrm{sgn}(a_{\phi(i)\phi(j)})$.
\end{theorem}

\begin{proof}
    Following Equation~\eqref{eq:1}, let $F$ denote the network flow and assume the dynamics in \eqref{eq:5}.
	We will first note that for smooth odd functions $f$, a convenient coordinate transformation yields unsigned dynamics.
	Let $z=G_tx$, or $z_i=\sigma_i x_i$.
	Then, following \cite{Altafini2012} we get the equivalent dynamics
	\begin{align*}
		\dot{z}_i&=-\sigma_i\sum_{j\in\mathcal{N}_i}f(\sigma_i z_i)-\mathrm{sgn}(a_{ij})f(\sigma_j z_j) \\
		&=-\sigma_i\sum_{j\in\mathcal{N}_i}\sigma_if(z_i)-\mathrm{sgn}(a_{ij})\sigma_jf(z_j) \\
		&=-\sum_{j\in\mathcal{N}_i}f(z_i)-f(z_j),
	\end{align*}
	which is an unsigned absolute nonlinear flow.
	Then, from Theorems \ref{thm:3} and \ref{thm:1} we conclude the system is inaccessible from the origin.
	
	Now, suppose $f$ is an even function. 
        From \eqref{eq:5} we have
	\begin{align*}
		F_i(\varphi'(x))&=-\sum_{l\in N_r}f(\sigma_i\sigma_rx_r)-\sgn(a_{ij})f(\sigma_j\sigma_lx_l) \\
		&=-\sum_{l\in N_r}f(x_r)-\sgn(a_{ij})f(x_l),
	\end{align*}
	where $r=\phi(i)$ and $l=\phi(j)$ and the property $f(\sigma_i\sigma_j x)=f(x)$ of even functions is used.
	
	On the other hand, we know from the above identity that\\ 
		$F_{\phi(i)}(x)=F_r(x)=-\sum_{l\in N_r}\left[f(x_r)-\sgn(a_{rl})f(x_l)\right].$
	Hence, $F$ is $\varphi'$-invariant if $\mathrm{sgn}(a_{ij})=\mathrm{sgn}(a_{rl})$.
\end{proof}
The condition on even functions in Theorem \ref{thm:4} can be interpreted as an edge-sign symmetry of the graph, in that edge signs remain invariant under the signed automorphism.
\subsubsection{Relative Nonlinear Flow}

The main result of this section shows that structural balance and the existence of the non-trivial signed automorphism lead to the uncontrollability of the relative nonlinear flow.

\begin{theorem}
	\label{thm:5}
	Consider a structurally balanced graph $\mathcal{G}$ with gauge transformation $G_t$ and \emph{relative nonlinear flow} dynamics.
	Further suppose $\mathcal{G}$ has a non-trivial automorphism $\varphi'$.
	Let $f:\mathbb{R}^n\to \mathbb{R}$ be a smooth odd or even function (for example, odd $f\in\mathcal{S}_0$ with $f$ smooth).
	Then, for any vertex $j\in\mathrm{Fix}(\varphi')$ chosen as the leader, the leader-follower network is not accessible from the origin in $\mathbb{R}^{n-1}$.
\end{theorem}
\begin{proof}
	Let the same notations as in proof of Theorem \ref{thm:4} hold.
	We will show that both cases of odd and even functions $f$ lead to $\varphi'$-invariance of the flow $F$ and therefore the inaccessibility from the origin.
	
	Let $f$ be an odd function.
        Changing the coordinates by $z=G_tx$ yields
	\begin{align*}
		\dot{z}_i&=-\sigma_i\sum_{j\in\mathcal{N}_i}f(\sigma_i z_i-\mathrm{sgn}(a_{ij})\sigma_j z_j) \\
		&=-\sigma_i\sum_{j\in\mathcal{N}_i}f(\sigma_i(z_i-\sigma_i\mathrm{sgn}(a_{ij})\sigma_j z_j)) \\
		&=-\sum_{j\in\mathcal{N}_i}f(z_i-z_j),
	\end{align*}
	which is the unsigned relative nonlinear flow.
	Hence, $F$ is $\varphi'$-invariant and inaccessibility from the origin follows from Theorems \ref{thm:3} and \ref{thm:1}.
	
	For an even function $f$, from \eqref{eq:6}
	\begin{align}
		&F_i(\varphi'(x))=-\sum_{l\in\mathcal{N}_{r}} f(\sigma_i\sigma_rx_r-\sigma_j\sigma_l\mathrm{sgn}(a_{ij})x_{l})\\
		&=-\sum_{l\in\mathcal{N}_{r}} f(\sigma_i\sigma_r(x_r-\sigma_r\sigma_l\sigma_i\sigma_j\mathrm{sgn}(a_{ij})x_{l}))\\
				&=-\sum_{l\in\mathcal{N}_{r}} f(x_r-\sigma_r\sigma_lx_{l})
				=-\sum_{l\in\mathcal{N}_{r}} f(x_r-\mathrm{sgn}(a_{rl})x_{l})
	\end{align}
	where the last display is equal to $F_r(x)$, and we have used the fact that $\sigma_i\sigma_j\mathrm{sgn}(a_{ij})>0$ hence $\sigma_i\sigma_j=\mathrm{sgn}(a_{ij})$ for all $i$ and $j\in\mathcal{N}_i$.
\end{proof}
\subsubsection{Disagreement Nonlinear Flow}

\begin{theorem}
	\label{thm:6}
	Consider a structurally balanced graph $\mathcal{G}$ with gauge transformation $G_t$ and \emph{disagreement nonlinear flow} dynamics.
	Further suppose $\mathcal{G}$ has a non-trivial automorphism $\varphi'$.
        Let $f:\mathbb{R}^n\to \mathbb{R}$ be a smooth odd or even function (for example, odd $f\in\mathcal{S}_0$ with $f$ smooth).
	Then, for any vertex $j\in\mathrm{Fix}(\varphi')$ chosen as the leader, the leader-follower network is not accessible from the origin in $\mathbb{R}^{n-1}$.
\end{theorem}

\begin{proof}
  The proof is identical to that of Theorem \ref{thm:5}.
\end{proof}


\begin{remark}
	The analysis of this section demonstrates that for all of the three nonlinear flows \eqref{eq:5}-\eqref{eq:7} structural balance in addition to $\varphi'$-invariance leads to system uncontrollability for even and odd functions $f$.
	The only exception is when the absolute nonlinear flow $f$ is even.
	In this case, an edge-sign symmetry condition is also required.
\end{remark}


\section{BIPARTITE IDENTIFICATION WITH THE KOOPMAN OPERATOR AND EDMD}
\label{sec:bipart-ident-with}

In this section, we extend the data-driven approach by Pan \emph{et al.} in using data-driven methods to identify the bipartite consensus for some of the nonlinear network flows considered in the preceding section.
Our main result asserts that the Koopman mode corresponding to the zero eigenvalue of the Koopman operator contains the sign structure indicating the bipartite consensus.

\begin{theorem}
  \label{thr:1}
  Consider either the dynamics \eqref{eq:5} with $f\in \mathcal{S}$ or the dynamics \eqref{eq:6} with $f\in\mathcal{S}_0$.
  Recall that the full-state observable can be written in terms of the Koopman triple as 
  \begin{align}
    x(t) = \sum_{j=1}^\infty e^{\lambda_j t} \varphi_j(x_0) v_j.\label{eq:2}
  \end{align}
  If the underlying graph is structurally balanced, then the following hold:
  \begin{enumerate}
    \item  $\lambda_j \leq 0$ with a unique zero eigenvalue $\lambda_1$.
    \item  The sign structure of the corresponding Koopman mode $v_1$ displays the bipartition of nodes denoted in Lemma~\ref{lem:1}(5).
  \end{enumerate}
  
\end{theorem}

\begin{proof}
  By Theorem~\ref{thr:2}, if $\mathcal{G}$ is structurally balanced, then we have that 
	$\lim_{t\to\infty}x(t) =n^{-1} \left(\mathbf{1}^T G_tx_0\right) G_t \mathbf{1}$.
        Hence, $\lambda_i \leq 0$ since otherwise the RHS of Equation~\eqref{eq:2} does not converge. 

The sign structure of the vector $G_t \mathbf{1}$ corresponds to bipartition of nodes denoted in Lemma~\ref{lem:1}(5).
Since $\lambda_1 = 0$ is unique, by setting $\alpha = (1/n) \mathbf{1}^T G_t x_0$ we have that
\begin{align}
  \lim_{t\to\infty} x(t) = \lim_{t\to\infty} \sum_{j=1}^\infty e^{\lambda_j t} \varphi_j(x_0) v_j = \varphi_1(x_0) v_1 = \alpha G_t \mathbf{1}.
\end{align}
Since both $\alpha$ and $\varphi_1(x_0)$ are constants, we can see that $v_1 \propto G_t \mathbf{1}$, and the result follows.
\end{proof}


In \S\ref{sec:ident-bipart-struct}, we show an example where we use EDMD to approximate the first mode of the Koopman operator to obtain the sign structure corresponding to the bipartite consensus.


\section{EXAMPLES}
\label{sec:examples}
In this section, we show some examples that highlight the necessity of combining $\varphi$-invariance, structural balance and leader-node symmetry for uncontrollability of signed networks.
We refer the reader to~\cite{Aguilar2014} for similar examples in the unsigned case.
We then show an example of using EDMD to obtain the bipartite consensus structure of a nonlinear flow on a structurally balanced graph.

\subsection{Unsigned Symmetry is Not Sufficient for Uncontrollability}
  Here we show an example of a network flow on a signed graph which has a leader-node symmetry.
  We show that in one case, the graph is structurally balanced, and the induced flow is hence uncontrollable.
  By altering the sign on a single edge, structural balance is lost and the resulting network flow is controllable.
  \begin{figure}[H]
  	\centering
  	\begin{minipage}{0.45\linewidth}
  		\centering
  		\subfloat[]{\label{fig:SB}\includegraphics[scale = 0.5]{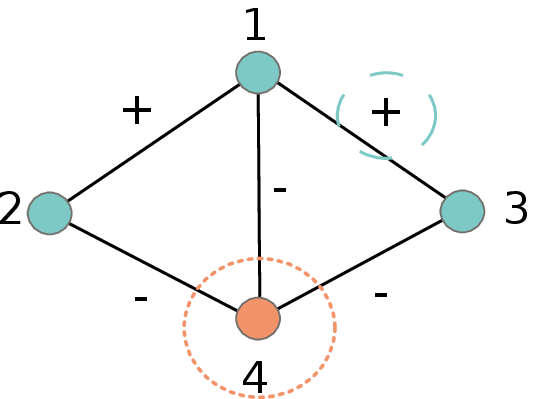}}
  	\end{minipage}
  	\hspace{3mm}
  	\begin{minipage}{0.45\linewidth}
  		\centering
  		\subfloat[]{\label{fig:SUB}\includegraphics[scale = 0.5]{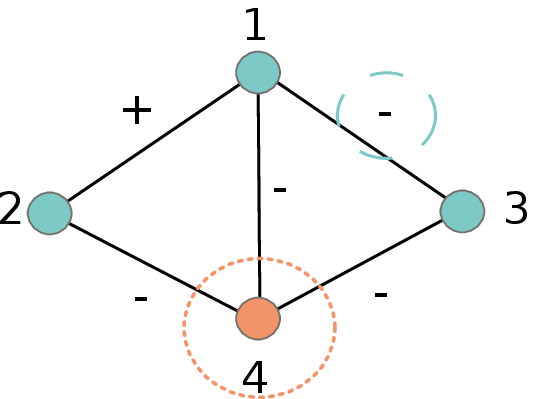}}
  	\end{minipage}
  	\caption{(a) Structurally balanced graph with a leader symmetry about node 4. (b) Structurally unbalanced graph with a leader symmetry about node 4.}
  	\label{fig:lap}
  \end{figure}
Consider the structurally balanced graph in Figure  \ref{fig:SB}, and the structurally unbalanced graph in Figure  \ref{fig:SUB}.
These have the dynamics
  $\dot x = -L_ix - \textbf{1}u$,
with Laplacians
\begin{align}
  L_{1} = \left[ 
  \begin{array}{ccc}
     2 &   -1 &    0 \\
    -1  &   3  &   1  \\
     0   &  1   &  2   
  \end{array}\right],~  L_{2} = \left[ 
  \begin{array}{ccc}
     2 &   -1 &    0 \\
    -1  &   3  &   -1  \\
     0   &  -1   &  2  
  \end{array}\right].
\end{align}
The controllability matrices of these dynamics are rank-deficient and full-rank, respectively.
 

\subsection{Identification of Bipartite Structure: Kuramoto Dynamics}\label{sec:ident-bipart-struct}
In this subsection, we consider a numerical method to identify the bipartite consensus structure of nonlinear dynamics on a structurally balanced graph, i.e. obtaining the bipartition of nodes in Lemma~\ref{lem:1}(5).
We do this by exploiting Theorem~\ref{thr:1}, and numerically approximating the Koopman mode corresponding to the zero eigenvalue.

Consider the dynamics \eqref{eq:6} with $f(\cdot) = \sin(\cdot)$, corresponding to the (signed) phase coupling of the Kuramoto dynamics\footnote{ The variable $x_i$ corresponds to the transformed phase coordinate $x_i = \phi_i - t\omega_0$, where $\omega_0$ is the oscillator's natural frequency. See Equation 2.2 in \cite{brown2003globally}}, with an underlying structurally balanced oscillator network shown in Figure~\ref{fig:edmdgraph}.
\begin{figure}[H]
  \centering
  \includegraphics[scale=0.75]{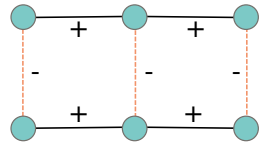}
  \caption{Underlying structurally balanced signed graph for EDMD example. Dashed edges indicate negative edges.}
\label{fig:edmdgraph}
\end{figure}
We use EDMD to approximate the first Koopman mode corresponding to the zero eigenvalue.
Due to space constraints, we refer the reader to \S2.2.3 of \cite{Williams2015} for a detailed explanation of the EDMD algorithm. 
In particular, we use a dictionary of functions of the form $\psi_k = \prod_{i=1}^6 H(x_i,j_i)$ where $H(x_i,j_i)$ is the $j_i$th order Hermite polynomial of $x_i$.
We use all such possible $\psi_k$ for Hermite polynomials up to order 2.
Simple combinatorics indicates that there are 729 such functions $\psi_k$ for a 6-node graph; hence $N_k=729$ in step (2) of \S2.2.3 in \cite{Williams2015}.
\begin{figure}[H]
	\centering
	\includegraphics[width=\columnwidth]{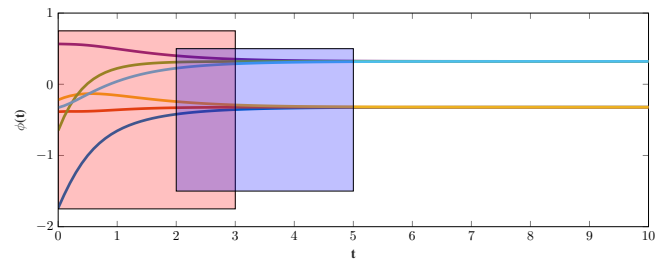}        
	\caption{Evolution of Dynamics~\ref{eq:6} with $f(\cdot)=\sin(\cdot)$ on the graph in Figure~\ref{fig:edmdgraph}. Left (red) shaded region corresponds to $\bar{v}_1^2$ and the right (blue) shaded region corresponds to $\bar{v}_1^3$.}
	\label{fig:data}
\end{figure}
The dynamics are shown in Figure~\ref{fig:data} for an initial condition of 
$x_0 = (   -1.73,
   -0.38,
   -0.21,
    0.56,
   -0.65,
   -0.32)$.
   The EDMD procedure was applied to three sets of data with a time-step of 0.1, as depicted in Figure~\ref{fig:data}: the full data (from $0\leq t \leq 10$), the data in the red shaded region ($0\leq t \leq 3)$ and the data in the blue shaded region $2\leq t \leq 5$).
   \newcommand{\numScale}{0.6}
The computed Koopman modes corresponding to $\bar{\lambda}_1\approx 1$ for these regions respectively are 
\begin{align}
\bar{v}_1^1 = {\scriptscriptstyle \left[ 
  \begin{array}{c}
\scalebox{\numScale}{    0.018}\\
\scalebox{\numScale}{    0.026}\\
\scalebox{\numScale}{    0.016}\\
\scalebox{\numScale}{   -0.020}\\
\scalebox{\numScale}{   -0.020}\\
\scalebox{\numScale}{   -0.011}    
  \end{array}\right]},~\bar{v}_1^2 = {\scriptscriptstyle \left[ 
  \begin{array}{c}
\scalebox{\numScale}{    0.019}\\
\scalebox{\numScale}{    0.021}\\
\scalebox{\numScale}{    0.015}\\
\scalebox{\numScale}{    -0.018}\\
\scalebox{\numScale}{    -0.020}\\
\scalebox{\numScale}{    -0.007}
  \end{array}\right]},~\bar{v}_1^3 = {\scriptscriptstyle \left[ 
  \begin{array}{c}
\scalebox{\numScale}{    0.016}\\
\scalebox{\numScale}{    0.023}\\
\scalebox{\numScale}{    0.015}\\
\scalebox{\numScale}{   -0.018}\\
\scalebox{\numScale}{   -0.016}\\
\scalebox{\numScale}{   -0.012}    
  \end{array}\right]},
\end{align}
which all contain sign structure corresponding to the bipartition depicted in Figure~\ref{fig:edmdgraph}.
Despite not using all available data, the EDMD procedure was able to extract the bipartition well before the dynamics converged.
This begs the question \emph{how early can we detect the bipartite structure of the underlying dynamics?}
We will address this question in future works.


\section{CONCLUSION}
\label{sec:concl-future-works}

In this paper, we examined the controllability of nonlinear flows on signed networks.
In particular, we identified that the topological property of \emph{structural balance} is the key ingredient that when combined with a leader-node symmetry and $\varphi$-invariance of the flow dynamics, results in uncontrollability.

We then looked at the task of identifying the bipartite consensus of certain classes of nonlinear network flows.
We showed that the sign structure of the first Koopman mode corresponds to this bipartite consensus, and then used EDMD to numerically approximate this sign structure.


\bibliographystyle{IEEEtran}
\bibliography{ConsensusKoopman}

\end{document}